\documentclass[reqno,10pt,a4paper,oneside]{amsart}
\usepackage{geometry}

\usepackage[english]{babel}
\usepackage[T1]{fontenc}

\usepackage[all,pdf]{xy}
\usepackage[babel]{csquotes}
\usepackage[export]{adjustbox}
\usepackage[ruled]{algorithm2e}
\usepackage{amscd}
\usepackage{amsmath}
\usepackage{amssymb}
\usepackage{amsthm}
\usepackage{bbm}
\usepackage{booktabs}
\usepackage{colortbl}
\usepackage{enumerate}
\usepackage{fancyhdr}
\usepackage{float}
\usepackage{graphicx}
\usepackage{hyperref}
\usepackage{indentfirst}
\usepackage{lmodern}
\usepackage{longtable}
\usepackage{mathrsfs}
\usepackage{mathtools}
\usepackage{mleftright}
\usepackage{pgfplots}
\pgfplotsset{compat=1.18}
\usepackage{tikz,tikz-cd}
\usepackage{xurl}

\usepackage[font=footnotesize]{caption}
\usepackage{subcaption}

\usetikzlibrary{arrows}
\usetikzlibrary{matrix}
\usetikzlibrary{shapes.misc}

\usepackage{etoolbox}
\apptocmd{\sloppy}{\hbadness 10000\relax}{}{}

\newcommand\undermat[2]{%
  \makebox[0pt][l]{$\smash{\underbrace{\phantom{%
    \begin{matrix}#2\end{matrix}}}_{\text{$#1$}}}$}#2}

\newcommand{\conv}{\mathrm{conv}}
\newcommand{\ehr}{\mathrm{ehr}}

\newcommand{\Vol}{\mathrm{Vol}}

\newcommand{\orow}{\rowcolor[gray]{0.95}}
\newcommand{\erow}{}

\newtheorem{theorem}{Theorem}[section]

\newtheorem{example}[theorem]{Example}
\newtheorem{lemma}[theorem]{Lemma}
\newtheorem{proposition}[theorem]{Proposition}
\newtheorem{question}[theorem]{Question}
\newtheorem{remark}[theorem]{Remark}

\theoremstyle{definition}

\date{}

\title{A genetic algorithm to search the space of Ehrhart $h^*$-vectors}

\author[G.~Balletti]{Gabriele Balletti}

\address{(G. Balletti)
Stockholm, Sweden
}
\email{gabriele.balletti@gmail.com}


\allowdisplaybreaks

\begin{document}

\begin{abstract}
    We describe a genetic algorithm to find candidates for $h^*$-vectors satisfying given properties in the space of integers vectors of finite length. 
    We use an implementation of such algorithm to find a 52-dimensional lattice polytope having a non-unimodal $h^*$-vector which is the Cartesian product of two lattice polytopes having unimodal $h^*$-vectors.
    This counterexample answers negatively to a question by Ferroni and Higashitani.
\end{abstract}

\maketitle

\section{Introduction}

A lattice polytope $P \subset \mathbb{R}^n$ is a convex polytope whose vertices are elements of the lattice $\mathbb{Z}^n$.
The map $t \mapsto |tP \cap \mathbb{Z}^n|$ for integer values of $t$ has been proven by Ehrhart \cite{ehrhart1962polyedres} to be a polynomial $\ehr_P(x) = c_d x^d + c_{d-1} x^{d-1} + \cdots + c_0$ of degree $d$, where $d$ is the dimension of $P$.
This is the \emph{Ehrhart polynomial} of $P$.
Through a change of basis, one can obtain the \emph{$h^*$-polynomial} of $P$ as the polynomial $h^*_P(x) = h^*_0 + h^*_1 x + \cdots + h^*_d x^d$, which is explicitly given by
\begin{equation}\label{eq:e_to_h}
    h^*_P(x) = \sum_{i=0}^{d} c_i A_i(x) (1-x)^{d-i},
\end{equation}
where $A_i(x)$ is the $i$-th Eulerian polynomial.
The inverse transformation, which allows to recover $\ehr_P$ from $h^*_P$, is given by
\begin{equation}\label{eq:h_to_e}
    \ehr_P(x) = \sum_{i=0}^d h^*_i \binom{x+d-i}{d}.
\end{equation}

The significance of these changes of bases comes from the algebraic interpretation of $\ehr_P$ and $h^*_P$ as the Hilbert function and polynomial of the Ehrhart ring associated to $P$.
See \cite{beck2007computing} for a comprehensive introduction to invariants of lattice polytopes and their algebraic meaning.

While the coefficients of the Ehrhart polynomial are rational numbers, the coefficients of the $h^*$-polynomial are nonnegative integers.
This follows from a combinatorial interpretation of the coefficients $h^*_i$, but holds in the more general settings of rational polytopes, thanks to a result by Stanley \cite{stanley1980decompositions}.
Moreover, it is easy to deduce that $h^*_0 = 1$, $h^*_1 = |P \cap \mathbb{Z}^n| - d - 1$, and $h^*_d = |P^\circ \cap \mathbb{Z}^n|$, where $P^\circ$ is the relative interior of $P$.
From the last two equalities it follows immediately that $h^*_1 \geq h^*_d$.
Other more profound inequalities have been proven over the years, we collect the most relevant ones by Stanley \cite{stanley1991hilbert}, Hibi \cite{hibi1994bound} and Stapledon \cite{stapledon2009inequalities} in Theorem~\ref{thm:conditions}, together with the previous one.

\begin{theorem}\label{thm:conditions}
    Let $P$ be a lattice polytope of dimension $d$ and let $s$ be the degree of its $h^*$-polynomial $h^*_P$.
    Then, the coefficients of $h^*_P$ satisfy the following inequalities:
    \begin{enumerate}
        \item $h^*_1 \geq h^*_d$,\label{item:condition_3}
        \item $h^*_2+h^*_3+\cdots+h^*_{i} \geq h^*_{d-1}+h^*_{d-2}+\cdots+h^*_{d-i+1}$ for every $2\leq i\leq \lfloor\frac{d}{2}\rfloor$,\label{item:condition_4}
        \item $h^*_0+h^*_1+\cdots+h^*_i \leq h^*_s+h^*_{s-1} + \cdots + h^*_{s-i}$ for every $0\leq i\leq \lfloor \frac{s}{2} \rfloor$,
        \item If $s=d$, then $h^*_1\leq h^*_i$ for all $1\leq i\leq d - 1$, and
        \item If $s\leq d-1$, then $h^*_0+h^*_1 \leq h^*_i + h^*_{i-1} + \cdots + h^*_{i-(d-s)}$ for all $1\leq i\leq d-1$.
    \end{enumerate}
\end{theorem}

All these inequalities are known to be sharp.
There are other general inequalities that have been proven in recent years, most notably the work of Stapledon \cite{stapledon2016additive}, but we will not consider them here.
Even considering all the known inequalities, we are still far from having a complete understanding of the set of all possible $h^*$-vectors of lattice polytopes, and even in dimension three the problem is still open (see \cite[Conjecture~8.7]{balletti2021enumeration}).

A set of inequalities between the coefficients of $h^*_P$ which is of particular interest, is the one defining unimodality.
We say that $h^*_P$ is \emph{unimodal} if there exists an index $i$ such that $h^*_0 \leq h^*_1 \leq \cdots \leq h^*_i \geq h^*_{i+1} \geq \cdots \geq h^*_d$.
While, for a general polytope $P$, unimodality of the $h^*$-polynomial is not guaranteed, a particular hierarchy of polytopes has attracted much attention in connection to this property, with a central role played by the so-called IDP polytopes.
A polytope $P$ is \emph{IDP} (short for has the \emph{integer decomposition property}) if for every positive integer $k$, any lattice point $p\in kP\cap\mathbb{Z}^n$ can be written as $p = p_1+\cdots+p_k$ where $p_1,\ldots,p_k\in P\cap \mathbb{Z}^n$.
Proving that the $h^*$-polynomial of an IDP polytope is unimodal is arguably one of today's most significant challenges in Ehrhart Theory.
It has been conjectured to be true -- in a more general setting -- by Brenti \cite{brenti1994combinatorics} in the nineties, while a direct formulation can be found in \cite{schepers2013unimodality}.
This challenge has been the object of recent developments, as a proof has been given by Adiprasito, Papadakis, Petrotou and Steinmeyer \cite{adiprasito2022beyond}, under the additional assumption of reflexivity.
It is not known if the IPD assumption can be relaxed, e.g. to very ample polytopes.
See the surveys by Braun \cite{braun2016unimodality} and by Ferroni and Higashitani \cite{ferroni_higashitani2023} for definitions, examples, and a detailed overview of the state of the art.

Unimodality of $h^*$-polynomial is often studied in conjunction with two stronger properties, namely log-concavity and real-rootedness.
We say that a degree $s$ polynomial $\sum_{i=0}^s a_i x^i$ is \emph{log-concave} if $a_i^2 \geq a_{i-1} a_{i+1}$ for all $1 \leq i \leq s-1$.
The same polynomial is said to be \emph{real-rooted} if all its roots are real.
We restrict these notions to polynomials with positive coefficients, i.e. $a_i > 0$ for all $0 \leq i \leq s$.
Under these assumptions, it is well known that log-concavity implies unimodality, and that real-rootedness implies log-concavity.

In this context, Ferroni and Higashitani focus on cartesian products of lattice polytopes, the idea being that being IDP, as well as other related properties in the aforementioned hierarchy, are preserved under cartesian products \cite[Proposition~2.2]{ferroni_higashitani2023}.
Moreover, it follows from a more general result of Wagner \cite{wagner1992total}, that the $h^*$-polynomial of the cartesian product of two polytopes with real-rooted $h^*$-polynomials is real-rooted too.
It is natural then to wonder if the log-concavity and the unimodality of the $h^*$-polynomial are preserved under cartesian products.
They do so in \cite[Question~3.4]{ferroni_higashitani2023}, expecting a positive answer for log-concavity and a negative answer for unimodality.
In this paper we focus on, and answer to, the unimodality question.

\begin{question}[{\cite{ferroni_higashitani2023}}]\label{question:main}
    Let $P$ and $Q$ be two lattice polytopes.
    If $h^*_P(x)$ and $h^*_Q(x)$ are unimodal, is $h^*_{P\times Q}(x)$ necessarily unimodal too?
\end{question}

A negative answer is given by constructing examples of lattice polytopes having a non-unimodal $h^*$-vector which are cartesian products of polytopes having unimodal $h^*$-vectors.
We also construct a lattice polytope having a non-unimodal $h^*$-vector which is the cartesian product of a polytope having unimodal $h^*$-vectors with itself.
This means that we cannot exclude cartesian products from the tools one might attempt to use to disprove unimodality conjectures.
The lowest dimension we can construct such a counterexample is 52.

\subsection{Operators between Ehrhart and \texorpdfstring{$h^*$}{}-polynomials}
One might attempt a direct approach to Question~\ref{question:main} by explicitly describing the $h^*$-polynomial of the cartesian product of two polytopes in terms of the $h^*$-polynomials of the two factors.
While it follows directly from the definitions that $\ehr_{P_1 \times P_2} = \ehr_{P_1} \ehr_{P_2}$ for any two polytopes $P_1$ and $P_2$, the description of the $h^*$-polynomial of the cartesian product is not as straightforward.
We can define two operators $\mathscr{W}$ and $\mathscr{E}$, each inverse of the other, mapping Ehrhart polynomials to the corresponding $h^*$-polynomials and vice versa.
Since the dimension $d$ of the polytope plays a role in these transformations, it is convenient to think of these operators as being defined between the spaces of vectors of finite length.
Here, each vector stores the coefficients of the corresponding polynomial up to degree $d$, regardless of the actual degree of the polynomial.
For the $h^*$-polynomial, the corresponding vector is the vector $(h^*_0, h^*_1, \ldots, h^*_d)$, while for the Ehrhart polynomial it is the vector $(c_0, c_1, \ldots, c_d)$.
The first of the two is known as the \emph{$h^*$-vector} of the polytope $P$, and it is denoted by $h^*_P$; note that it can have trailing zeroes.
We can use Equations~\eqref{eq:e_to_h} and \eqref{eq:h_to_e} to extend $\mathscr{W}$ and $\mathscr{E}$ to the space of all real vectors of finite length, which we write as the disjoint union $\bigsqcup_{i > 0} \mathbb{R}^i$:
\begin{align*}
    \mathscr{W}: \bigsqcup_{i > 0} \mathbb{R}^i  &\to \bigsqcup_{i > 0} \mathbb{R}^i \\
    (c_0, c_1, \ldots, c_d) &\mapsto \sum_{i=0}^{d} c_i A_i(x) (1-x)^{d-i},
\end{align*}
and
\begin{align*}
    \mathscr{E}: \bigsqcup_{i > 0} \mathbb{R}^i  &\to \bigsqcup_{i > 0} \mathbb{R}^i \\
    (h^*_0, h^*_1, \ldots, h^*_d) &\mapsto \sum_{i=0}^d h^*_i \binom{x+d-i}{d}.
\end{align*}
If we now define the binary operator
\begin{equation}\label{eq:prod}
    \begin{aligned}
        \Pi: \bigsqcup_{i > 0} \mathbb{R}^i \times \bigsqcup_{i > 0} \mathbb{R}^i  &\to \bigsqcup_{i > 0} \mathbb{R}^i \\
        (h, h') &\mapsto \mathscr{W}(\mathscr{E}(h) \mathscr{E}(h')), 
    \end{aligned}
\end{equation}
then we have a compact way to express the $h^*$-vector of a cartesian product:
\[
    h^*_{P_1 \times P_2} = \Pi(h^*_{P_1}, h^*_{P_2}).
\]

While this description allows one to concretely calculate the $h^*$-polynomial of the cartesian product when the two factors are given, it is very hard to deduce general statements about the unimodality of $h^*$-polynomials from it, partly because explicitly expanding the operators $\mathscr{W}$ and $\mathscr{E}$ is nontrivial when the dimensions in which we operate are not fixed, partly because we would have to consider all possible ways in which the $h^*$-polynomials of the factors and of the product can be unimodal.

For these reasons, we take a different approach to build our way to a counterexample.
We start from the following example from \cite{ferroni_higashitani2023}.

\begin{example}[{\cite[Example~3.5]{ferroni_higashitani2023}}]\label{ex:ferroni_higashitani}
    The vectors $h = (1, 1, 1, 1, 1, 6)$ and $h' = (1, 1, 1, 1, 2, 5)$ are unimodal, while $\Pi(h, h') = (1, 38, 300, 962, 2059, 7442, 7194, 7292, 4320, 854, 30)$ is not.
\end{example}

Note that it is not a counterexample to Question~\ref{question:main}, as the vectors involved cannot be $h^*$-vectors of lattice polytopes, as they violate inequality \ref{item:condition_3} of Theorem~\ref{thm:conditions}. 
We use this example (and a slight variant of it, Example~\ref{ex:ferroni_higashitani_bis}), as a starting point for a genetic algorithm designed to explore the space of nonnegative integer vectors of finite length, searching for possible counterexamples to Question~\ref{question:main}.
The idea is to keep mutating the vectors $h$ and $h'$ until they become ``more and more similar to $h^*$-vectors of lattice polytopes'', in the hope that they will eventually become $h^*$-vectors of lattice polytopes.
We do so by defining a fitness function which penalizes vectors depending on how much they violate the inequalities of Theorem~\ref{thm:conditions}.
This needs to be done while ensuring that $h$ and $h'$ remain unimodal, while $\Pi(h, h')$ remains non-unimodal.

\subsection{Structure of the paper}
In Section~\ref{sec:methodology} we give a brief introduction to genetic algorithms, as well as set the notation and terminology for the actual implementation.
Section~\ref{sec:implementation} is the core of this paper, as it contains details on the implementation of the algorithm, and the results obtained.
In Section~\ref{sec:realizing} we show how to realize the vectors found by the algorithm as $h^*$-vectors of lattice polytopes, thus answering Question~\ref{question:main} negatively.
A final section, Section~\ref{sec:conclusion}, includes some concluding remarks and possible future developments.

\subsection{Acknowledgements}
This project started in response to an unexpected invitation to give a talk at the KTH Combinatorics Seminar in Stockholm, after several years away from academia.
I am very grateful to Luis Ferroni for providing this opportunity, and for double checking some computations present in this paper.
Both he and Akihiro Higashitani provided valuable feedback on a preliminary version of this paper.
I am grateful to the two anonymous referees whose suggestions helped refine this work.

\section{An introduction to genetic algorithms}\label{sec:methodology}

Genetic algorithms were formalized by Holland in the 1960s.
The original intent was not that of designing a process to solve a specific problem, but rather building a formal framework to study evolutionary phenomena as they occur in nature.
In his work, summarized in \cite{holland1992adaptation}, Holland describes genetic algorithms as an abstraction of biological evolution, where genetic operators such as mutation, crossover are used to simulate the natural selection of individuals in a population, and where the fitness of an individual is determined by its ability to survive and reproduce.
Nowadays, genetic algorithms are used in a wide variety of fields, mainly as a tool to solve optimization problems in which traditional methods are not applicable, or are not efficient enough.
See the introduction by Mitchell \cite{mitchell1998introduction} for a detailed overview of genetic algorithms and their use.

Today there is no universally accepted definition of what a genetic algorithm is, and there are many different variants of the same basic idea.
Often, the term \emph{evolutionary algorithm} is used to refer to a broader class of algorithms which includes genetic algorithms and are inspired by the same biological principles.

In this section we give a brief description of what a genetic algorithm is, trying to both tailor it to our needs, but also remaining quite faithful to the original formulation by Holland.
This acts both as a reference for the implementation described in Section~\ref{sec:implementation}, and as way to fix notation and terminology.

\subsection{Genetic algorithms in a nutshell}

The goal of a genetic algorithm is to look for certain solutions in a set, which we call the \emph{search space}.
We model the search space with a directed weighted graph on a (possibly infinite) vertex set $V$, and with edges having weights defined by a function $\mu: V \times V \to \mathbb{R}_{\geq 0}$.
The edges are the pairs $(v, v') \in V \times V$ such that $\mu(v, v') > 0$.

Elements of $V$ are called \emph{individuals} or \emph{chromosomes}, and a subset $U \subset V$ is referred to as a \emph{population}.

The directed edges define the possible \emph{mutations} between pairs of individuals, with probabilities given by the weights.
For example, the directed edge $e = v_1 \to v_2$ indicates that the individual $v_1$ can mutate into the individual $v_2$ with probability $\mu(v_1, v_2)$. With an abuse of notation we indicate by $\mu(v)$ the operation of sampling an individual from the distribution described by $\mu(v, \cdot)$ starting from an initial individual $v$.

Two individuals $v_1, v_2 \in V$ can generate a new individual $u \in V$ with probabilities given by a \emph{crossover operation} $\chi: (V \times V) \times V \to \mathbb{R}_{\geq 0}$, where $\chi((v_1, v_2), u)$ describes the probability that the individual $u$ is generated by crossing over the individuals $v_1$ and $v_2$.
In general, $\chi$ is assumed to be a symmetric operation in the first two coordinates, i.e. $\chi((v_1, v_2), \cdot) = \chi((v_2, v_1), \cdot)$.
With another abuse of notation, we will write $\chi(v_1, v_2)$ to indicate the operation of sampling from the distribution $\chi((v_1, v_2), \cdot)$.
Intuitively, $\mu$ and $\chi$ are expected to be defined in some way that the ``offsprings'' inherits properties from their ``parents''.

While mutation and crossover can be used to grow the size of a population, an evolution process also requires a way to reduce its size.
This is done by a \emph{selection operator} $\sigma_\varphi: 2^V \to 2^V$, where $\sigma_\varphi(U)$ is the subset of individuals of $U$ which are selected to survive.
Selection is driven by a \emph{fitness function} $\varphi: V \to \mathbb{R}_{\geq 0}$, which evaluates the fitness of each individual, where lower values indicate higher fitness.
Exactly how the selection is done, and specifically how $\varphi$ is defined, depends on the specific context, but in general the idea is to select individuals with lower fitness values, and remove the others.

The goal of the algorithm is to find elements $\varphi^{-1}(0) \subset V$, these are the \emph{solutions} of the algorithm.
Once a solution is found, the algorithm terminates and returns it.
It is clear that $\varphi$ needs to satisfy some notion of regularity or continuity, so that it is possible to approach the solutions set $S$ by iteratively improving the fitness of the individuals in the population.

The algorithm starts from an initial population $U \subset V$ of individuals, and iteratively let evolution act on them through mutations, crossover and selection, until optimal solutions are found.
In general there is no guarantee that the algorithm will terminate, so a maximum number of iterations $T$ is set in advance.
This procedure is described in Algorithm~\ref{alg:general}.

\begin{algorithm}[htbp]
    \SetKwInOut{Input}{input}
    \SetKwInOut{Output}{output}
    \Input{Initial population $U$}
    \Output{Optimal solutions}
    \For{$t \gets 1$ \KwTo $T$}{
        Repeatedly apply $\mu$ and $\chi$ to individuals in $U$\;
        \If{$0 \in \varphi(U)$}{
            \Return{$\{u \in U \mid \varphi(u) = 0 \}$} \;
        }    
        $U \gets \sigma_\varphi(U)$\;
    }
    \caption{General procedure for an evolutionary algorithm.}
    \label{alg:general}
\end{algorithm}

Note how the fitness function $\varphi$ is the cornerstone of the algorithm.
It selects which individuals are allowed to survive, and it determines when the algorithm terminates.
Note that in other applications $\varphi$ might never evaluate to zero for any individual, and the algorithm might be forced to terminate when a local minimum is found or after a certain number of iterations.

\section{Evolving vectors into \texorpdfstring{$h^*$}{}-vectors}\label{sec:implementation}

In this section we specialize Algorithm~\ref{alg:general} to search for a counterexample to Question~\ref{question:main}.
The main idea here is to define a search space as the pairs of vectors of nonnegative integers of finite length satisfying certain unimodality-related properties.
On this space and to define a fitness function $\varphi$ which penalizes individuals for ``not looking enough like $h^*$-vectors of lattice polytopes'', in the hope that vectors will be driven into evolving into $h^*$-vectors of lattice polytopes.
We formalize this by defining the following properties that a pair of vectors $(h, h')$ can have.
We make use of the operator $\Pi$, previously defined in Equation~\eqref{eq:prod}.
\begin{enumerate}[i.]
    \item $h$ and $h'$ are both unimodal,\label{item:cond_unimodal}
    \item $\Pi(h, h')$ is not unimodal,\label{item:cond_product}
    \item $h$ and $h'$ are $h^*$-vectors of some lattice polytopes.\label{item:cond_realizable}
\end{enumerate}
Any pair of vectors satisfying these properties would give a counterexample to Question~\ref{question:main}.

We will define the search space as the set of all possible vectors of nonnegative integers of finite length satisfying properties~\ref{item:cond_unimodal} and~\ref{item:cond_product}.
The verification of properties \ref{item:cond_unimodal} and~\ref{item:cond_product} is computationally inexpensive, making it possible for the population to evolve quickly through the search space.
Moreover, Example~\ref{ex:ferroni_higashitani} guarantees that this set is not empty, giving us candidates to start the search with.

Conversely, checking property~\ref{item:cond_realizable} is an exceptionally hard problem in Ehrhart Theory.
While algorithms to check whether a given vector is the $h^*$-vector of a lattice polytope can be theoretically implemented through an exhaustive search \cite{balletti2021enumeration}, no efficient algorithm is known to date.
Consequently, it is unclear how to set up a fitness function which would allow to approach pairs of vectors satisfying property~\ref{item:cond_realizable}.
We know however that the $h^*$-vectors of lattice polytopes need to fulfill all the conditions from Theorem~\ref{thm:conditions}, so we can use these to define a fitness function which penalizes vectors which do not satisfy them, depending on how much they violate the conditions.
Although this is not a way to certify property~\ref{item:cond_realizable}, we can use it to obtain a large catalogue of candidates for being a counterexample to Question~\ref{question:main}, largely simplifying the task.
We describe this fitness function in Subsection~\ref{subsec:selection}.

The specialized algorithm can be setup in two different variants, each with its own advantages and disadvantages.
\begin{enumerate}[A]
    \item We can look for individual elements $h$ such that $(h, h)$ satisfies properties~\ref{item:cond_unimodal}--~\ref{item:cond_realizable}.
        This introduces a strong bias in the search, possibly reducing its potential to find solutions, but making the algorithm more efficient.\label{item:h} 
    \item We can look for pairs of elements $(h, h')$ such that $(h, h')$ satisfies properties~\ref{item:cond_unimodal}--~\ref{item:cond_realizable}.
        This allows to explore a large and unbiased search space, but the search will be computationally expensive.\label{item:hh}
\end{enumerate}

For simplicity we focus on variant~\ref{item:h}, as its implementation and notation will be much lighter.
Anyway its generalization \ref{item:hh} can be described by taking the cartesian product of two copies of the search space, extending the mutations and crossover operations to act independently on each copy, and defining the fitness function to be the sum of the fitness functions of the two copies.
The source code, which is available on GitHub\footnote{\url{https://github.com/gabrieleballetti/genetic-search}}, contains details about both the implementations.

\subsection{Search space}
Let $\mathcal{H}$ be the set of all possible vectors of nonnegative integers of finite length satisfying properties~\ref{item:cond_unimodal} and~\ref{item:cond_product}, i.e.
\[
    \mathcal{H} \coloneq \{ h \in \sqcup_{i > 0} \mathbb{Z}_{\geq 0}^i \mid \text{$h$ is unimodal and $\Pi(h, h)$ is not} \}.
\]
We can slightly modify Example~\ref{ex:ferroni_higashitani} to show that $\mathcal{H}$ is not empty.

\begin{example}\label{ex:ferroni_higashitani_bis}
    The pair $(h,h)$ with $h = (1, 1, 1, 1, 1, 6)$ satisfies properties~\ref{item:cond_unimodal} and~\ref{item:cond_product}.
    Indeed $h$ is unimodal, while $\Pi(h, h) = (1, 38, 300, 962, 1849, 7417, 7034, 7272, 4610, 973, 36)$ is not.
\end{example}

As for Example~\ref{ex:ferroni_higashitani}, the vector $h$ above cannot be the $h^*$-vector of a lattice polytope.

\subsection{Mutation}
We make the search space $\mathcal{H}$ into a graph by defining mutations, which will act as the directed edges.
For an element $h = (h_0, h_1, \ldots, h_d) \in \mathcal{H}$ we define the following mutation operations:
\begin{itemize}
    \item the \emph{removal} mutations $\mu^{\operatorname{rem}}_{i}$, for all $0 \leq i \leq d$, where 
        \[
            \mu^{\operatorname{rem}}_{i}(h) = (h_0, \ldots, h_{i-1}, h_{i+1}, \ldots, h_d),
        \]
    \item the \emph{insertion} mutations $\mu^{\operatorname{ins}_a}_{i}$, for all $0 \leq i \leq d$, and $0 \leq a \leq \max_i (h)$, where
        \[
            \mu^{\operatorname{ins}_a}_{i}(h) = (h_0, \ldots, h_{i}, a, h_{i+1}, \ldots, h_d),
        \]
    \item the \emph{decrement} mutations $\mu^{-}_{i}$, for all $0 \leq i \leq d$, where
        \[
            \mu^{-}_{i}(h) = (h_0, \ldots, h_{i-1}, h_i - 1, h_{i+1}, \ldots, h_d),
        \]
    \item the \emph{increment} mutations $\mu^{+}_{i}$ for all $0 \leq i \leq d$, where
        \[
            \mu^{+}_{i}(h) = (h_0, \ldots, h_{i-1}, h_i + 1, h_{i+1}, \ldots, h_d).
        \]
\end{itemize}

We exclude all mutations which would result in a vector which is not in $\mathcal{H}$, for example the decrement of a zero entry, or any insertion which would result in a vector which is not unimodal.
We denote by $M(h)$ the set of all possible mutations for an individual $h$, and we set for each of them weight $\frac{1}{|M(h)|}$, meaning that applying the mutation operator $\mu$ we will get any of the possible mutations in $M(h)$ with equal probability.
Note that, a priori, it is not trivial to determine if $\mathcal{H}$ is connected, or even if $h$ from Example~\ref{ex:ferroni_higashitani_bis} has any neighbors.
This means that at each step we will check which mutations are still in $\mathcal{H}$, out of all the possible ways to mutate an individual $h$, and only consider those.

\begin{example}
    Let $h = (1, 1, 1, 1, 1, 6) \in \mathcal{H}$ be the vector from Example~\ref{ex:ferroni_higashitani_bis}.
    The following are the neighbors of $h$, i.e. all the allowed mutations of $h$:
    \begin{align*}
        \mu^{\operatorname{ins}_1}_{i}(h) &= (1, 1, 1, 1, 1, 1, 6), \text{ for $0 \leq i \leq 4$}\\
        \mu^{\operatorname{ins}_2}_{4}(h) &= (1, 1, 1, 1, 1, 2, 6), \\
        \mu^{\operatorname{ins}_3}_{4}(h) &= (1, 1, 1, 1, 1, 3, 6), \\
        \mu^{+}_{5}(h) &= (1, 1, 1, 1, 1, 7).
    \end{align*}
    This means that $\mu(h)$ will return one of above vectors with equal probability $\frac{1}{4}$.
\end{example}

\subsection{Crossover}
The crossover operation is defined as follows.
Let $h = (h_0, h_1, \ldots, h_d)$ and $h' = (h'_0, h'_1, \ldots, h'_{d'})$ be two elements of $\mathcal{H}$, and assume $d \geq d'$.
We define their crossover operation $\chi(h, h')$ to return -- with equal probability -- any of the vectors of the form $g = (g_0, g_1, \ldots, g_d) \in \mathcal{H}$ such that for each $0 \leq i \leq d'$ we have either $h_i \leq g_i \leq h'_i$ or $h'_i \leq g_i \leq h_i$, and
$g_i = h_i$ whenever $i > d'$.
Although this definition might seem quite arbitrary, we found empirically that crossover greatly improves performance and convergence of the algorithm, as shown in Figure~\ref{fig:crossover}.
Generally, one can see crossover as a way to allow individuals trapped in local minima for the fitness function to still play a role in the evolution of the population.

\subsection{Selection}\label{subsec:selection}
We define the fitness function $\varphi: \mathcal{H} \to \mathbb{R}_{\geq 0}$ to be the measure of how much an element $h = (h_0, h_1, \ldots, h_d)$ violates the conditions from Theorem~\ref{thm:conditions}.
Specifically, given any individual $h \in \mathcal{H}$, $\varphi(h)$ is the sum of the violations of each condition divided by $d$, i.e. 
\[
    \varphi(h) \coloneq \frac{1}{d}\left(\varphi_1(h) + \varphi_2(h) + \varphi_3(h) + \varphi_4(h)\right),
\] 
where  
\begin{enumerate}
    \item $\varphi_1(h) \coloneq \max(0, h_d - h_1)$,
    \item $\varphi_2(h) \coloneq \sum_{i = 2}^{\lfloor \frac{d}{2} \rfloor} \max(0, h_{d-i+1} + \cdots + h_{d-1} - (h_{2} + \cdots + h_{i}))$,
    \item $\varphi_3(h) \coloneq \sum_{i = 0}^{\lfloor \frac{s}{2} \rfloor} \max(0, h_0 + \cdots + h_i - (h_s + \cdots + h_{s-i}))$,
    \item $\varphi_4(h) \coloneq \left \{
        \begin{array}{l}
            \sum_{i = 1}^{d-1} \max(0, h_i - h_1), \text{ if } s = d, \\
            \sum_{i = 1}^{d-1} \max(0, h_0 + h_1 - (h_i + h_{i-1} + \cdots + h_{i-(d-s)})), \text{ if } s < d.\\
        \end{array} \right .$
\end{enumerate}
Above, $s$ is the index of the last nonzero entry of $h$.
If $h$ happens to be the $h^*$-vector of a lattice polytope $P$, then $s$ would be the degree of the $h^*$-polynomial of $P$.
Dividing the sum by $d$ has a normalizing effect on the fitness function, and allows to interpret the value of $\varphi(h)$ (up to a constant) as the average violation of the inequalities of Theorem~\ref{thm:conditions}.
This ensures that the algorithm does not penalize long vectors excessively.
Note that $h$ is a solution, i.e. $h \in \varphi^{-1}(0)$, if and only if $h$ satisfies all the conditions of Theorem~\ref{thm:conditions}.

\subsection{Population management}
The population at a given step of the algorithm -- we call such step a \emph{generation} -- will be increased through mutations and crossover, and reduced through selection.
Note however that mutation and crossover have opposite effects on the diversity of the population.
Indeed, mutation tends to increase the diversity of the population, while crossover tends to decrease it.
In order to control how the diversity evolves over time, we set probabilities $p_\mu$ and $p_\chi$, such that $p_\mu + p_\chi = 1$, and we use them to decide whether to apply mutation or crossover to the population every time that we increase population by one.
At each generation, we artificially increase the population until a certain constant $N_\text{max}$ is reached, and then we apply selection to reduce it to  $N_\text{min}$.

\subsection{Avoiding local minima}
It is possible for individuals to sit on local minima for the fitness function $\varphi$.
Crossover helps such individuals to still play a role in the evolution of the population, but eventually these individuals might become attractors for future generations, slowing down the global convergence towards a solution.
In order to avoid this, we introduce a \emph{penalty factor} $\varepsilon \in \mathbb{R}_{>1}$ to penalize individuals which have been in the population for too many generations.
Specifically, for any individual $h\in \mathcal{H}$ we will multiply the output of the fitness function $\varphi(h)$ by $\varepsilon^{\operatorname{age}(h)}$, where $\operatorname{age}(h)$ is the number of generations that $h$ has been in the population.
As a result, the selection process will tend to favor younger individuals.
This is an important step for the convergence of the algorithm: as shown in Figure~\ref{fig:evolution_h} the starting individual needs to evolve through states with a higher fitness value, before being able to evolve into a solution.

\subsection{Pseudocode}
The pseudocode for the algorithm is presented in Algorithm~\ref{alg:implementation}.
We use the following two auxiliary functions to make the pseudocode more readable.
\begin{itemize}
    \item $\operatorname{increase}_{p_\mu}(U)$ is the operation of increasing the population $U$ by one individual, by applying the following steps:
        \begin{enumerate}[1.]
            \item randomly choose whether to apply mutation or crossover, depending $p_\mu$,
            \item pick one individual (two for crossover) randomly from $U$,
            \item apply the chosen operation, $\mu$ or $\chi$, to the individual(s),
            \item add the resulting individual to $U$.
        \end{enumerate}
    \item $\operatorname{age}(h)$ is the number of generations that $h$ has been in the population.
\end{itemize}

\begin{algorithm}[htbp]
    \SetKwInOut{Input}{input}
    \SetKwInOut{Output}{output}
    \SetKwInOut{Parameters}{parameters}
    \Parameters{
        $p_\mu \in [0,1]$, probability of applying mutation instead of crossover, \\
        $N_\text{max} \in \mathbb{N}$, upper limit for the population, \\
        $N_\text{min} \in \mathbb{N}$, lower limit for the population \\
        $\varepsilon \in \mathbb{R}_{>1}$, penalty factor, \\
        $T \in \mathbb{N}$, maximum number of iterations.
    }
    \Input{$U \subset \mathcal{H}$, initial population}
    \Output{Optimal solutions}
    \For{$t \gets 1$ \KwTo $T$}{
        \While{$|U| < N_\text{max}$}{
            $\operatorname{increase}_{p_\mu}(U)$
        }
        \If{$0 \in \varphi(U)$}{
            \Return{$\{h \in U \mid \varphi(h) = 0 \}$} \;
        }
        \While{$|U| > N_\text{min}$}{
            $h_{\max} \gets \arg \max_{h \in U} \varphi(h) \varepsilon^{\operatorname{age}(h)} $\;  
            $U \gets U \setminus \{h_{\max}\}$\;
        }
    }
    \caption{Evolutionary search algorithm on $\mathcal{H}$ (variant~\ref{item:h}).}
    \label{alg:implementation}
\end{algorithm}

\subsection{Further exploration of the solutions space}\label{subsec:further_exploration}
Once one or multiple solutions are found, it is possible to generate new solutions by iteratively mutating the individuals found so far, and adding them to the population.
If we consider the induced subgraph of the search space given by the subset of solutions $\varphi^{-1}(0) \subset \mathcal{H}$, then this procedure is equivalent to a breadth-first exploration of the subgraph, or at least of the connected component containing the first solution found.
Since there is no obvious way to check if a given $h \in \mathcal{H}$ is the $h^*$-vector of a lattice polytope, having a way to generate new solutions will help us to find realizable vectors.

\subsection{Runs and results}
Algorithm~\ref{alg:implementation} is run multiple times with parameters $N_\text{max} = 30$, $N_\text{min} = 5$, $p_\mu = 0.5$, $p_\chi = 0.5$, $\varepsilon = 1.05$, 
and initial population $U$ set to $U = \{(1, 1, 1, 1, 1, 6)\}$.

All the runs converged to one or more solutions within the first 100 generations.
We can plot the average fitness value of the best individual in the population -- which is the one with the lowest fitness score -- at each generation, to observe how the population changes over time.
In Figure~\ref{fig:crossover} we show such plots for ten different runs, showing the difference between ``normal'' runs ($p_\mu = 0.5$), and runs without any crossover ($p_\mu = 1$).

\begin{figure}[htbp]
    \centering
    \subfloat[$p_\mu = 0.5$]{\includegraphics[width=0.5\linewidth]{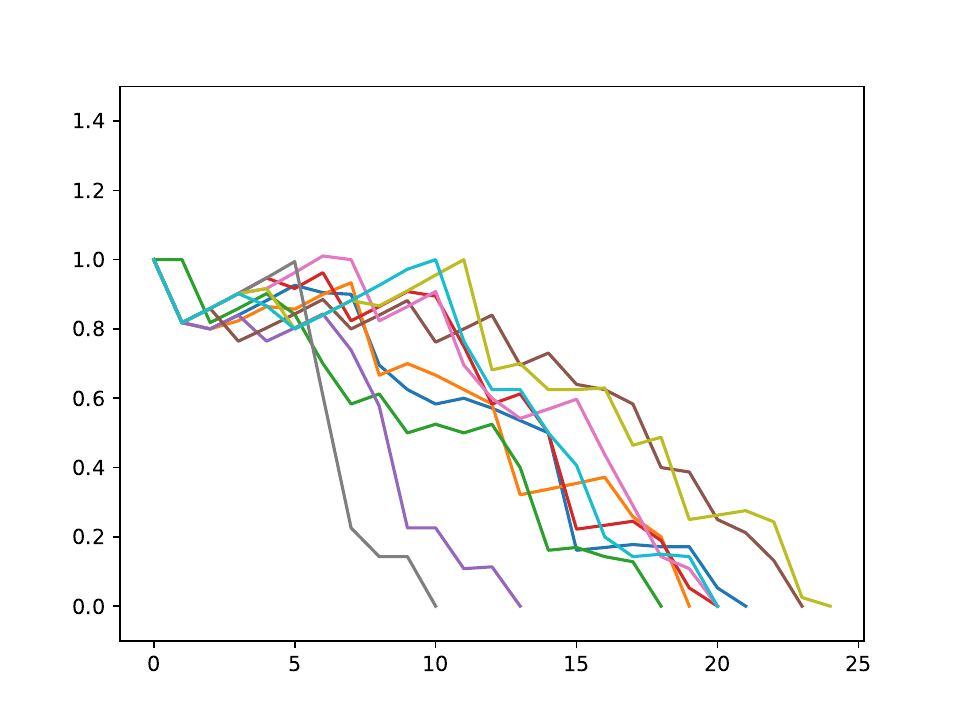}}%
    \subfloat[$p_\mu = 1$]{\includegraphics[width=0.5\linewidth]{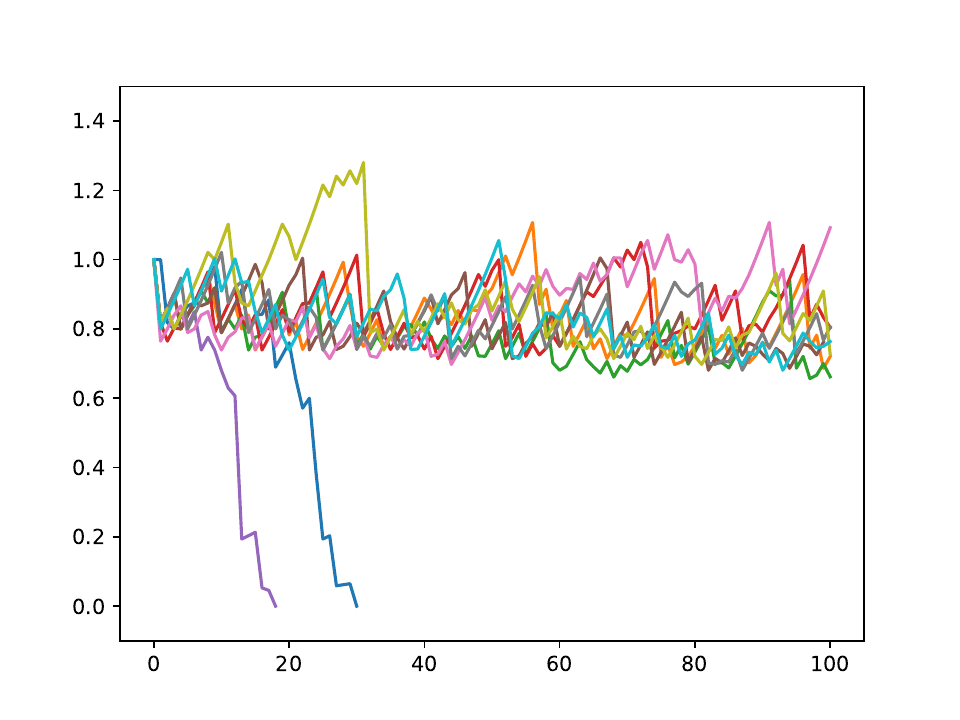}}%
    \caption{
        Average population fitness value (y axis) at each generation (x axis) over ten different runs of Algorithm~\ref{alg:implementation} with (left) and without (right) crossover.
        Most of the runs without crossover did not converge to a solution within the first 100 generations and were interrupted.
        The runs with crossover had a 50\% chance of applying mutation and a 50\% chance of applying crossover.
        Other parameters: $N_\text{max} = 30$, $N_\text{min} = 5$, $\varepsilon = 1.05$.
    }
    \label{fig:crossover}
\end{figure}

Picking a very small population will result in longer runs, but it allows us to observe the evolution of a single individual in more detail, as shown in Figure~\ref{fig:evolution_h}.
It is interesting to notice that the initial individuals are in a local minimum for the fitness function, and they need to evolve through states with a higher fitness value before being able to evolve into a solution.
This suggests that a greedy approach to the problem would have most likely failed.

\begin{figure}[htbp]
    \centering
    \includegraphics[width=0.75\textwidth]{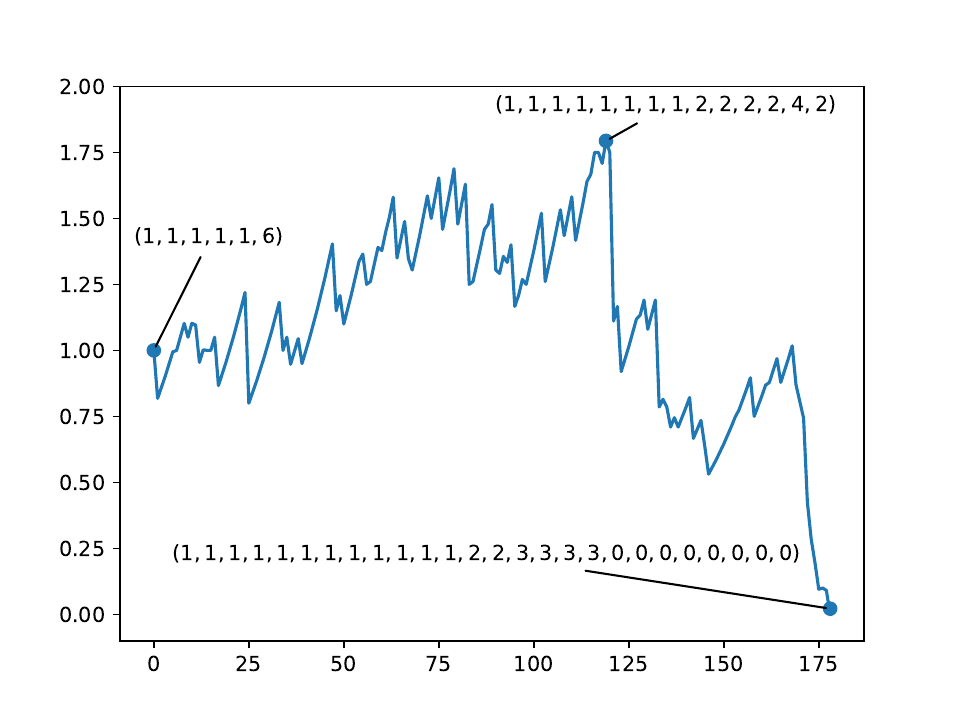}
    \caption{
        A plot of the fitness values (y axis) for the evolution of the starting individual $(1, 1, 1, 1, 1, 6)$ over $178$ generations (x axis), before it reaches a solution.
        Plot obtained with $N_\text{max} = 3$, $N_\text{min} = 1$, $p_\mu = 0.5$, $p_\chi = 0.5$, $\varepsilon = 1.05$.
    }
    \label{fig:evolution_h}
\end{figure}

The set of all solutions found are then expanded to 1000 elements each with the procedure described in Subsection~\ref{subsec:further_exploration}, ending up with a list of almost 5000 elements.
Of them, the shorter vector has length 22, and all of them have at least ten leading ones and at least seven trailing zeros.
See Table~\ref{tab:samples} for a random sample of solutions found by running Algorithm~\ref{alg:implementation}.
The full list of solutions found is available together with the source code of the implementation.

\begin{center}
    \small
    \begin{longtable}{l}
        \label{tab:samples}\\
        \toprule
        \endfirsthead
        \orow $(1, 1, 1, 1, 1, 1, 1, 1, 1, 1, 1, 1, 3, 7, 2, 0, 0, 0, 0, 0, 0, 0, 0, 0)$\\
        \erow $(1, 1, 1, 1, 1, 1, 1, 1, 1, 1, 1, 1, 1, 3, 5, 5, 0, 0, 0, 0, 0, 0, 0, 0, 0, 0)$\\
        \orow $(1, 1, 1, 1, 1, 1, 1, 1, 1, 1, 1, 1, 1, 1, 9, 3, 0, 0, 0, 0, 0, 0, 0, 0, 0, 0, 0)$\\
        \erow $(1, 1, 1, 1, 1, 1, 1, 1, 1, 1, 1, 1, 1, 2, 4, 7, 1, 0, 0, 0, 0, 0, 0, 0, 0, 0, 0)$\\
        \orow $(1, 1, 1, 1, 1, 1, 1, 1, 1, 1, 1, 1, 1, 2, 5, 4, 2, 2, 0, 0, 0, 0, 0, 0, 0, 0, 0)$\\
        \erow $(1, 1, 1, 1, 1, 1, 1, 1, 1, 1, 1, 1, 1, 3, 3, 5, 4, 0, 0, 0, 0, 0, 0, 0, 0, 0, 0)$\\
        \orow $(1, 1, 1, 1, 1, 1, 1, 1, 1, 1, 1, 1, 2, 2, 3, 3, 6, 0, 0, 0, 0, 0, 0, 0, 0, 0, 0)$\\
        \erow $(1, 1, 1, 1, 1, 1, 1, 1, 1, 1, 1, 1, 1, 1, 2, 5, 4, 2, 1, 0, 0, 0, 0, 0, 0, 0, 0, 0)$\\
        \orow $(1, 1, 1, 1, 1, 1, 1, 1, 1, 1, 1, 1, 1, 2, 2, 4, 4, 4, 0, 0, 0, 0, 0, 0, 0, 0, 0, 0)$\\
        \erow $(1, 1, 1, 1, 1, 1, 1, 1, 1, 1, 1, 1, 1, 2, 3, 3, 6, 1, 0, 0, 0, 0, 0, 0, 0, 0, 0, 0)$\\
        \orow $(1, 1, 1, 1, 1, 1, 1, 1, 1, 1, 1, 1, 1, 1, 1, 2, 5, 6, 1, 1, 0, 0, 0, 0, 0, 0, 0, 0, 0, 0)$\\
        \erow $(1, 1, 1, 1, 1, 1, 1, 1, 1, 1, 1, 1, 1, 1, 2, 2, 2, 5, 4, 1, 1, 1, 0, 0, 0, 0, 0, 0, 0, 0)$\\
        \orow $(1, 1, 1, 1, 1, 1, 1, 1, 1, 1, 1, 1, 1, 1, 2, 2, 2, 3, 4, 4, 2, 0, 0, 0, 0, 0, 0, 0, 0, 0, 0)$\\
        \erow $(1, 1, 1, 1, 1, 1, 1, 1, 1, 1, 1, 1, 1, 1, 2, 2, 2, 2, 2, 5, 4, 1, 1, 0, 0, 0, 0, 0, 0, 0, 0, 0)$\\
        \orow $(1, 1, 1, 1, 1, 1, 1, 1, 1, 1, 1, 1, 1, 1, 1, 1, 2, 3, 4, 4, 2, 1, 1, 0, 0, 0, 0, 0, 0, 0, 0, 0, 0)$\\
        \bottomrule
        \caption{A random sample of 15 solutions found by running Algorithm~\ref{alg:implementation} (variant~\ref{item:h})}
    \end{longtable}
\end{center}

The fact that solutions were found proves the following nontrivial result.

\begin{theorem}\label{thm:first}
    The set of integer vectors in $\mathcal{H}$ satisfying all conditions from Theorem~\ref{thm:conditions} is not empty.
\end{theorem}

Still, it is not clear whether any of the solutions found so far is the $h^*$-vector of a lattice polytope.
We proceed by selecting a solution with a peculiar shape: the pair $(h,h)$ with
\begin{equation}\label{eq:solutions_h}
    \begin{aligned}
        &h  = (1, 1, 1, 1, 1, 1, 1, 1, 1, 1, 1, 1, 1, 1, 12, 0, 0, 0, 0, 0, 0, 0, 0, 0, 0, 0, 0, 0) \in \mathbb{Z}^{28}.
    \end{aligned}
\end{equation}
In Section~\ref{sec:realizing} we will show that it is the $h^*$-vector of a lattice polytope.
Once we know that there are solutions $(h, h)$, with
\begin{equation}\label{eq:solution_general}
    h = (1, \underbrace{1, \ldots, 1}_{k-1}, m, \underbrace{0, \ldots, 0}_{k-1}) \in \mathbb{Z}^{2k},
\end{equation}
for some positive integers $k$ and $m$, we can focus our efforts on finding solutions of this form, with the goal of finding solutions as small as possible.
Unfortunately, a grid search for small values of $k$ and $m$ does not lead to any smaller solutions.

\begin{proposition}\label{prop:extra_solutions}
    The following are all the solutions of the form $(h, h)$ with $h$ as in \eqref{eq:solution_general}, with $k \leq 20$ and $m \leq 100$:
    \begin{itemize}
        \item $k = 14$, $m = 12$, which is the solution \eqref{eq:solutions_h},
        \item $k = 15$, $m = \{12, 13\}$,
        \item $k = 16$, $m = \{13, 14\}$, 
        \item $k = 17$, $m = \{13, 14, 15, 16\}$,
        \item $k = 18$, $m = \{13, 14, 15, 16, 17\}$,
        \item $k = 19$, $m = \{14, 15, 16, 17, 18, 19, 20\}$,
        \item $k = 20$, $m = \{14, 15, 16, 17, 18, 19, 20, 21, 22\}$.
    \end{itemize}
\end{proposition}
It is interesting to notice that fixing the length of the vectors with $k$ seems to restrict the possible values of $m$ to a finite interval.
On the other hand, we can now also look for solutions of the form $(h, h')$ with both $h$ and $h'$ being in the form \eqref{eq:solution_general}, but not equal to each other.
Another grid search for small values of $k$ and $m$ allows us to reduce the dimension of the minimal solution, but only slightly.
Indeed the pair $(h, h')$, with
\begin{equation}\label{eq:solutions_hh}
    \begin{aligned}
        &h  = (1, 1, 1, 1, 1, 1, 1, 1, 1, 1, 1, 1, 1, 14, 0, 0, 0, 0, 0, 0, 0, 0, 0, 0, 0, 0)  \in \mathbb{Z}^{26},\\
        &h' = (1, 1, 1, 1, 1, 1, 1, 1, 1, 1, 1, 1, 1, 1, 9, 0, 0, 0, 0, 0, 0, 0, 0, 0, 0, 0, 0, 0) \in \mathbb{Z}^{28},
    \end{aligned}
\end{equation}
satisfies all the conditions from Theorem~\ref{thm:conditions}, and $\Pi(h, h')$ is not unimodal.   

\section{Realizing \texorpdfstring{$h^*$}{}-vectors}\label{sec:realizing}

In this section we focus on finding lattice polytopes whose $h^*$-vectors are the solutions found in the previous section, thus answering Question~\ref{question:main} negatively.
We do so by finding a family of lattice polytopes in dimension $d = 2k - 1$ whose $h^*$-vectors are of the form
\[
    (1, \underbrace{r, \ldots, r}_{k-1}, r + q, \underbrace{0, \ldots, 0}_{k-1}),	
\]
where $r$ and $q$ can be any nonnegative integers, and $k \geq 2$.
An important family of simplices on which we base our construction is the \emph{generalized Reeve simplex} by Batyrev and Hofscheier.

\begin{proposition}[{\cite{batyrev_hofscheier_22}}]\label{prop:reeve}
    Let $d=2k-1$ for $k\geq 2$, and $q$ be a nonnegative integer.
    Consider the generalized Reeve simplex in dimension $d$
    \[
        R_{q,k}:=\operatorname{conv}\{0, e_1, \ldots, e_{d-1}, u_{q, k} \},
    \]
    where $u_{q, k} = (\underbrace{1,\ldots,1}_{k-1},\underbrace{q,\ldots,q}_{k-1},q+1)$.
    Then we have
    \[ 
        h^*_{R_{q,k}} = (1, \underbrace{0, \ldots, 0}_{k-1} , q, \underbrace{0, \ldots, 0}_{k-1}).
    \]
\end{proposition}

Stanley's monotonicity theorem~\cite{stanley1993monotonicity} guarantees that, if we build a lattice polytope by adding vertices to $R_{q,k}$, then its ``spike'' on the $k$-th entry will be preserved, meaning that any polytope containing $R_{q,k}$ will have $h^*$-polynomial with $k$-th coefficient greater than or equal to $q$.
In Proposition~\ref{prop:gentleman} we show that it is possible to add a point to $R_{q,k}$ and obtain a polytope with the desired $h^*$-vector. 
For its proof we need to introduce properties of certain point configurations called circuits, we refer to \cite{de2010triangulations} for a more detailed introduction to the topic, and to \cite{blanco2016lattice, balletti2021enumeration} for applications to the study of lattice polytopes.

We say that a point configuration $\mathcal{A}$ given by the points $p_0, \ldots, p_{d+1} \in \mathbb{R}^d$ is a \emph{circuit} if it satisfies a unique (up to scalar multiplication) affine dependence $\sum_{i = 0}^{d+1} \lambda_i p_i = 0$,
where $\lambda_i \neq 0$ for all $0 \leq i \leq d+1$ and $\sum_{i = 0}^{d+1} \lambda_i = 0$.
We denote by $J_+$ and $J_-$ the sets of indices of the positive and negative coefficients of the dependence, i.e.
\[
    J_+ = \{ v_i \in \mathcal{A} \mid \lambda_i > 0 \}, \quad J_- = \{ v_i \in \mathcal{A} \mid \lambda_i < 0 \},
\]
the partition of $\mathcal{A}$ into $J_+$ and $J_-$ is called the \emph{Radon partition} or the \emph{oriented circuit} of the circuit.
The following lemma lists well-known properties of polytopes obtained as convex hulls of circuits, see \cite[{Lemma~2.4.2 and Remark~4.1.8}]{de2010triangulations} for the proofs and further details.

\begin{lemma}\label{lem:circuits}
    Let $\mathcal{A} = v_0, \ldots, v_{d+1} \in \mathbb{R}^d$ be a circuit satisfying the affine linear dependence $\sum_{i = 0}^{d+1} \lambda_i v_i = 0$.
    Let moreover $\mathcal{A} = J_+ \cup J_-$ be its Radon partition.
    Then the $d$-dimensional polytope $P_\mathcal{A} \coloneq \operatorname{conv}(v_0, \ldots, v_{d+1})$ satisfies the following properties.
    \begin{enumerate}[(a)] 
        \item $P_\mathcal{A}$ has exactly two triangulations on the vertices $v_0, \ldots, v_{d+1}$, namely
        \[
            \mathcal{T}_+\coloneq\left\{C \subset \mathcal{A} : J_+ \not\subseteq C\right\},
            \qquad \text{and} \qquad
            \mathcal{T}_-\coloneq\left\{C \subset \mathcal{A} : J_- \not\subseteq C\right\}.
        \]\label{item:triangulations}
        \item There is a constant $\alpha \in \mathbb{R}_{>0}$ such that $\alpha|\lambda_i|$ equals the normalized volume $\operatorname{Vol}(S_i)$ of the full-dimensional simplex $S_i \coloneq \operatorname{conv}(v_0, \ldots, v_{i-1}, v_{i+1}, \ldots, v_{d+1})$ for all $0 \leq i \leq d+1$.\label{item:volume}
            In particular, $P_\mathcal{A}$ has normalized volume
            \[
                \operatorname{Vol}(P_\mathcal{A}) = \alpha\sum_{v_i \in J_+} \lambda_i = -\alpha\sum_{v_i \in J_-} \lambda_i.
            \]
    \end{enumerate}
\end{lemma}

\begin{proposition}\label{prop:gentleman}
    Let $d=2k-1$ for $k\geq 2$, and $q, r$ be nonnegative integers.
    Consider the lattice polytope
    \[
        P_{q,r,k}:=\operatorname{conv}(R_{q,k} \cup \{ v_{r, k} \}) = \operatorname{conv}\{0, e_1, \ldots, e_{d-1}, u_{q, k}, v_{r, k} \},
    \]
    with $u_{q, k} = (\underbrace{1,\ldots,1}_{k-1},\underbrace{q,\ldots,q}_{k-1},q+1)$ and $v_{r, k} = (\underbrace{0,\ldots,0}_{k-1},\underbrace{-r,\ldots,-r}_{k})$.
    Then we have
    \[ 
        h^*_{P_{q,r,k}} = (1, \underbrace{r, \ldots, r}_{k-1} , r + q, \underbrace{0, \ldots, 0}_{k-1}).
    \]
\end{proposition}
\begin{proof}
    We assume $r > 0$, as the case $r = 0$ is just Proposition~\ref{prop:reeve}.
    We denote by $\mathcal{A}$ the point configuration $v_0, \ldots, v_{d+1} \in \mathbb{R}^d$ the vertices of $P_{q,r,k}$, where the $v_i$ appear in the same order as the corresponding vertices appear in Proposition~\ref{prop:reeve}, and $v_{d+1} = v_{r, k}$.
    Note that $\mathcal{A}$ is a circuit, as it satisfies the affine dependence $\sum_{i = 0}^{d+1} \lambda_i v_i = 0$, where
    \[
        (\lambda_0, \ldots, \lambda_{d+1}) = (-q-r-1, \underbrace{-r, \ldots, -r}_{k-1}, \underbrace{r, \ldots, r}_{k}, q + 1) \in \mathbb{R}^{d+2},
    \]
    By Lemma~\ref{lem:circuits}(\ref{item:triangulations}) we know that $P$ has a triangulation $\mathcal{T}_+$ whose full-dimensional simplices are $S_k, \ldots, S_{d+1}$, where $S_i = \operatorname{conv}(v_0, \ldots, v_{i-1}, v_{i+1}, \ldots, v_{d+1})$.
    By Lemma~\ref{lem:circuits}(\ref{item:volume}), the normalized volume of $P_{q,r,k}$ is $\alpha(kr + q + 1)$, for some positive $\alpha \in \mathbb{R}_{>0}$.
    By verifying that $\operatorname{Vol}(S_{d+1}) = q + 1$ we can deduce that $\alpha = 1$, from which we get $\Vol(P_{q,r,k}) = kr + q + 1$.

    Note that the edge between $v_0$ and $v_{d+1}$ contains the $r-1$ (non-vertices) lattice points of the form $(0, \ldots, 0, -a, \ldots, -a)$ for all $1 \leq a \leq r - 1$.
    We can then deduce that $|P_{q,r,k} \cap \mathbb{Z}^d| \geq d + 2 + (r-1)$, where $d+2$ is the number of vertices of $P_{q,r,k}$.
    From this it follows that $h^*_{P_{q,r,k}, 1} = |P_{q,r,k} \cap \mathbb{Z}^{d}| - d - 1 \geq r$.

    We now verify that the $P_{q,r,k}$ contains a unimodular copy of the Reeve simplex $R_{q + r,k}$.
    We do this by finding the following unimodular linear map $A$ mapping $R_{q + r,k}$ to $S_0 -v_{d+1}$, where $S_0 = \conv(v_1, \ldots, v_{d+1}) \subset P_{q,r,k}$:
    {\footnotesize
    \[
        A =
        \begin{pmatrix}
            1      &        & 0      & -r   & \cdots & -r   & -r     \\
                   & \ddots &        & -r   & \cdots & -r   & -r     \\
            0      &        & 1      & -r   & \cdots & -r   & -r     \\
            0      & \cdots & 0      & -r+1 &        & -r   & -r     \\
            \vdots &        & \vdots &     & \ddots &     & \vdots   \\
            0      & \cdots & 0      & -r   &        & -r+1 & -r     \\
            \undermat{k-1}{0  & \cdots & 0 } & \undermat{k-1}{(k-1)r & \cdots & (k-1)r } & (k-1)r + 1 \\
        \end{pmatrix}.\\
        \vspace{1.5em}
    \]
    }%
    Since $h^*_{R_{q + r,k}}(x) = 1 + (q + r) x^k$, it follows that $h^*_{P_{q,r,k}, k} \geq q + r$.
    Note that, although one can verify that $A$ is indeed unimodular, the inequality above follows simply from $A$ being injective, as $A$ maps a full-dimensional simplex to another full-dimensional simplex.

    We conclude by noting that $kr + q + 1 = \Vol(P_{q,r,k}) \geq \sum_{i = 0}^k h^*_{P_{q,r,k}, i} \geq 1 + (k-1)r + (q + r) = kr + q + 1$, which implies that all the lower bounds for the $h^*$ coefficients we have found above are equalities.
\end{proof}

From the solutions \eqref{eq:solutions_h} and \eqref{eq:solutions_hh} we found in Section~\ref{sec:implementation}, together with Proposition~\ref{prop:gentleman}, we can finally find two counterexamples to Question~\ref{question:main}, the smallest of which is in dimension 52.

\begin{theorem}\label{thm:counterexamples}	
    There exists a 52-dimensional lattice polytope having a non-unimodal $h^*$-vector which is the cartesian product of two lattice polytopes having unimodal $h^*$-vectors.
    There exists a 54-dimensional lattice polytope having a non-unimodal $h^*$-vector which is the cartesian product of a lattice polytope having unimodal $h^*$-vector with itself.
\end{theorem}
\begin{proof}
    Let $P_1 = P_{13,1,13} \subset \mathbb{R}^{25}$, and $P_2 = P_{8,1,14} \subset \mathbb{R}^{27}$, then
    \begin{align*}
        h^*_{P_1} &= (1, 1, 1, 1, 1, 1, 1, 1, 1, 1, 1, 1, 1, 14, 0, 0, 0, 0, 0, 0, 0, 0, 0, 0, 0, 0), \text{ and }\\
        h^*_{P_2} &= (1, 1, 1, 1, 1, 1, 1, 1, 1, 1, 1, 1, 1, 1, 9, 0, 0, 0, 0, 0, 0, 0, 0, 0, 0, 0, 0, 0).
    \end{align*}
    Their cartesian product $P_1 \times P_2 \subset \mathbb{R}^{52}$ has $h^*$-vector
    {\footnotesize
    \begin{align*}
        h^*_{P_1 \times P_2} = (&1, 730, 124309, 8765488, 323507917, 7103108746, 100619559775, 972997522804,\\
        &6701090224708, 34010137321612, 131025667114366, 393954147455920, 950686811060974,\\
        &1897069464954708, 3236287404712418, 4916467171316240, 7021322555787755,\\
        &9969874216848215, 14257154161508525, 19582563191473835, 24471834266882555,\\
        &27411039808716545, 28264217617840835, \underline{28180199302315925}, 28214521222404290,\\
        &28017151283317190, 25859926811345216, 20532901649723562, 13208986918862508, \\
        &6616676890501554, 2511085966199400, 705968725977246, 143760309232842\\
        &20655352730313, 2024932293684, 129413526255, 5049065826, 108582147, 1075968\\
        &3285, 0, 0, 0, 0, 0, 0, 0, 0, 0, 0, 0, 0, 0),
    \end{align*}
    }%
    where the underlined value $h^*_{23}$ is strictly smaller than the previous and following values.
    Let instead $Q = P_{11,1,14} \subset \mathbb{R}^{27}$.
    Then,
    \begin{align*}
        h^*_{Q} &= (1, 1, 1, 1, 1, 1, 1, 1, 1, 1, 1, 1, 1, 1, 12, 0, 0, 0, 0, 0, 0, 0, 0, 0, 0, 0, 0, 0),
    \end{align*}
    The cartesian product $Q \times Q \subset \mathbb{R}^{54}$ has $h^*$-vector
    {\footnotesize
    \begin{align*}
        h^*_{Q \times Q} = (&1, 786, 144455, 11020300, 441121770, 10530181640, 162552269110, 1716525078180,\\
        &12930814578275, 71850371688370, 302985944870565, 995261924725160, 2613021081859780,\\
        &5632695584460000, 10262874607660101, 16321723300358178, 23556670816935360, \\
        &32503659106949550, 45236238837124725, 63958867580029500, 86973762823550460, \\
        &107371181389097340, 118749821455655325, 121429882556886150, \underline{121087083599905800}\\
        &122150570332920090, 122520628829356317,114289161415862640, 91957387766542420,\\
        &60298585492287800, 31048542899221005,12240036657236410, 3619520872541915,\\
        &786880755821820, 122905067056750, 13400888684480, 981797783010, 45828567940,\\
        &1259997095, 17987346, 105649, 144, 0, 0, 0, 0, 0, 0, 0, 0, 0, 0, 0, 0, 0),
    \end{align*}
    }%
    where the underlined value $h^*_{24}$ is strictly smaller than the previous and following values.
\end{proof}

\begin{remark}
    The first polytope described in Theorem~\ref{thm:counterexamples} is, dimensionally speaking, the smallest counterexample we could find with our method.
    The latter is also the smallest, given the additional constraint of being the cartesian product of a lattice polytope with itself.
    Anyway, they are not the only existing counterexamples, as each of the solutions listed in Proposition~\ref{prop:extra_solutions} corresponds to another such polytope.
    The increasing number of possible values of $m$ fixing the length of the vectors seems to indicate that not only an infinite family exists, but these polytopes get more common in higher dimensions.
    While trying to define such a family by explicitly expanding Equation~\eqref{eq:prod} in the case of vectors of the form $h^*_{P_{q,1,k}}$ seems possible, such an approach has not been explored in this work.
\end{remark}

It is worth mentioning that these cartesian products, although spanning, are not IDP nor very ample.
In particular, they are not counterexamples to the unimodality conjecture in these cases.

\section{Future directions and challenges}\label{sec:conclusion}

While now we know that unimodality of $h^*$-vectors of lattice polytopes is not guaranteed when taking the product of polytopes with unimodal $h^*$-vectors, we still do not know what can be said about log-concavity.
This means that \cite[Question~3.4~(b)]{ferroni_higashitani2023} is still open. 
Although one might think that the techniques used in this paper could be used to try to approach that question, a strong impediment is that there is no equivalent of Example~\ref{ex:ferroni_higashitani} for log-concavity, which means the same strategy cannot be used to try to investigate that question, as we have no base case to initialize a genetic algorithm with.
One might try to tweak a fitness function to specifically try to force evolution towards log-concave $h^*$-vectors, but beside the technical difficulties that this might imply, there is belief that a counterexample in the case of log-concavity might not exist, see the related discussion in \cite{ferroni_higashitani2023}.

The most immediate question that stems from this work is whether the cartesian product is a viable tool to construct counterexamples to the unimodality conjecture for IDP or very ample polytopes.
\begin{question}
    \leavevmode
    \begin{enumerate}[(a)]
        \item Are there lattice polytopes $P_1$ and $P_2$ which are IDP, such that the $h^*$-polynomial of $P_1 \times P_2$ is not unimodal?
        \item Are there lattice polytopes $P_1$ and $P_2$ which are very ample, such that the $h^*$-polynomial of $P_1 \times P_2$ is not unimodal?
    \end{enumerate}
\end{question}
Note that a positive answer to the first question would mean a counterexample to the unimodality conjecture for IDP polytopes, as such we expect that the answer is negative.
The second question, which is a natural generalization of the first, might be harder to answer, but it is important to remark that the $h^*$-vectors found in this paper (of which a random sample is shown in Table~\ref{tab:samples}) are most certainly not the right candidates for building such examples, due to their peculiar shapes and the restrictions that these impose on the geometry of the polytopes.
For this reason it is reasonable to believe that the answer to the second question is also negative. 

Finally, we believe that the family of polytopes $P_{q,r,k}$ constructed in Proposition~\ref{prop:gentleman} could be a rich source of examples for unimodality and log-concavity questions in Ehrhart Theory.
As an example, Theorem~\ref{thm:counterexamples} shows that they can used to build nontrivial examples of spanning polytopes with non-unimodal $h^*$-vectors, while $P_{3,6,107}$ can be used to replicate the counterexample from \cite[Theorem~5.5]{ferroni_higashitani2023} of a unimodal $h^*$-polynomial whose Ehrhart series is not log-concave.

\bibliographystyle{amsalpha}
\bibliography{bibliography}

\providecommand{\bysame}{\leavevmode\hbox to3em{\hrulefill}\thinspace}
\providecommand{\MR}{\relax\ifhmode\unskip\space\fi MR }
\providecommand{\MRhref}[2]{%
  \href{http://www.ams.org/mathscinet-getitem?mr=#1}{#2}
}
\providecommand{\href}[2]{#2}
\begin{thebibliography}{DLRS10}

\bibitem[APPS22]{adiprasito2022beyond}
Karim~Alexander Adiprasito, Stavros~Argyrios Papadakis, Vasiliki Petrotou, and
  Johanna~Kristina Steinmeyer, \emph{Beyond positivity in {E}hrhart {T}heory},
  arXiv preprint arXiv:2210.10734 (2022).

\bibitem[Bal21]{balletti2021enumeration}
Gabriele Balletti, \emph{Enumeration of lattice polytopes by their volume},
  Discrete \& Computational Geometry \textbf{65} (2021), no.~4, 1087--1122.

\bibitem[BH22]{batyrev_hofscheier_22}
Victor Batyrev and Johannes Hofscheier, \emph{A generalization of a theorem of
  {W}hite}, Moscow Journal of Combinatorics and Number Theory \textbf{10}
  (2022), no.~4, 281--296.

\bibitem[BR07]{beck2007computing}
Matthias Beck and Sinai Robins, \emph{Computing the continuous discretely},
  vol.~61, Springer, 2007.

\bibitem[Bra16]{braun2016unimodality}
Benjamin Braun, \emph{Unimodality problems in {E}hrhart theory}, Recent trends
  in combinatorics (2016), 687--711.

\bibitem[Bre94]{brenti1994combinatorics}
Francesco Brenti, \emph{Combinatorics, and {G}eometry: an update}, Jerusalem
  Combinatorics' 93: An International Conference in Combinatorics, May 9-17,
  1993, Jerusalem, Israel, vol. 178, American Mathematical Soc., 1994, p.~71.

\bibitem[BS16]{blanco2016lattice}
M{\'o}nica Blanco and Francisco Santos, \emph{Lattice 3-polytopes with few
  lattice points}, SIAM Journal on Discrete Mathematics \textbf{30} (2016),
  no.~2, 669--686.

\bibitem[DLRS10]{de2010triangulations}
Jes{\'u}s De~Loera, J{\"o}rg Rambau, and Francisco Santos,
  \emph{Triangulations: structures for algorithms and applications}, vol.~25,
  Springer Science \& Business Media, 2010.

\bibitem[Ehr62]{ehrhart1962polyedres}
Eugene Ehrhart, \emph{Sur les poly{\`e}dres rationnels homoth{\'e}tiques {\`a}
  n dimensions}, CR Acad. Sci. Paris \textbf{254} (1962), 616.

\bibitem[FH24]{ferroni_higashitani2023}
Luis Ferroni and Akihiro Higashitani, \emph{Examples and counterexamples in
  {E}hrhart theory}, EMS Surv. Math. Sci. (2024), to appear.

\bibitem[Hib94]{hibi1994bound}
T.~Hibi, \emph{A lower bound theorem for {E}hrhart polynomials of convex
  polytopes}, Advances in Mathematics \textbf{105} (1994), no.~2, 162--165.

\bibitem[Hol92]{holland1992adaptation}
John~H Holland, \emph{Adaptation in natural and artificial systems: an
  introductory analysis with applications to biology, control, and artificial
  intelligence}, MIT press, 1992.

\bibitem[Mit98]{mitchell1998introduction}
Melanie Mitchell, \emph{An introduction to genetic algorithms}, MIT press,
  1998.

\bibitem[Sta80]{stanley1980decompositions}
Richard~P Stanley, \emph{Decompositions of rational convex polytopes}, Ann.
  Discrete Math \textbf{6} (1980), no.~6, 333--342.

\bibitem[Sta91]{stanley1991hilbert}
\bysame, \emph{On the {H}ilbert function of a graded {C}ohen-{M}acaulay
  domain}, journal of pure and applied algebra \textbf{73} (1991), no.~3,
  307--314.

\bibitem[Sta93]{stanley1993monotonicity}
\bysame, \emph{A monotonicity property of $h$-vectors and $h^*$-vectors},
  European Journal of Combinatorics \textbf{14} (1993), no.~3, 251--258.

\bibitem[Sta09]{stapledon2009inequalities}
Alan Stapledon, \emph{Inequalities and ehrhart $\delta$-vectors}, Transactions
  of the American Mathematical Society \textbf{361} (2009), no.~10, 5615--5626.

\bibitem[Sta16]{stapledon2016additive}
\bysame, \emph{Additive number theory and inequalities in {E}hrhart theory},
  International Mathematics Research Notices \textbf{2016} (2016), no.~5,
  1497--1540.

\bibitem[SVL13]{schepers2013unimodality}
Jan Schepers and Leen Van~Langenhoven, \emph{Unimodality questions for
  integrally closed lattice polytopes}, Annals of Combinatorics \textbf{17}
  (2013), 571--589.

\bibitem[Wag92]{wagner1992total}
David~G Wagner, \emph{Total positivity of hadamard products}, Journal of
  mathematical analysis and applications \textbf{163} (1992), no.~2, 459--483.

\end{thebibliography}

\end{document}